\documentclass{amsart}
\usepackage{lipsum}
\usepackage[dvips]{graphicx}
\usepackage{amssymb,latexsym}
\usepackage[ansinew]{inputenc}
\usepackage{mathrsfs}
\usepackage{psfrag}
\usepackage[all]{xy}
\usepackage{color}

\usepackage{eufrak}
\usepackage{amsfonts}
\usepackage{float}
\usepackage{amsmath,amssymb}
\usepackage{lscape}

\newtheorem{proposition}{Proposition}[section]
\newtheorem{theorem}{Theorem}

\newtheorem{remark}{Remark}[section]

\newcommand{\A}{\mathcal{A}}

\newcommand{\R}{\mathbb{R}}

\def\qed{\hfill $\Box$}
\def\proof{\noindent {\sl Proof} :\;  }

\def\b0{\mbox{\boldmath $0$}}

\begin{document}

\title{PROJECTIVE CLASSIFICATION OF JETS OF SURFACES IN $\mathbb{P}^4$}
\author{JORGE LUIZ DEOLINDO SILVA AND YUTARO KABATA}

\maketitle
%\textbf{Adviser:} Toru Ohmoto

\begin{abstract}We are interested in the local extrinsic geometry of smooth surfaces in 4-space, 
and classify  jets of Monge forms by projective transformations according to $\A^3$-types of their central projections.
\end{abstract}

\section{Introduction}
We are concerned with the geometry of the contact of smooth surfaces in the projective space $\mathbb{R}P^4$ with their tangent lines.
%local generic extrinsic geometry of smooth surfaces in the projective space $\mathbb{R}P^4$.
The contact is measured by types of map germs of {\it central projections} of the surfaces.
%Here $\A$ means the product of the sets of diffeormorphism germs of the source and the target which naturally acts on the space of smooth map germs,
%and $\A^\ell$ is the space of $\ell$-jets of elements in $\A$ which acts on the $\ell$-jet space of map germs.
%the equivalence of map germs through the actions of diffeomorphism-germs of the source and the target 
%and $\A^\ell$
In the present paper we classify jets of generic surfaces by projective transformations which preserve the geometry of the contacts of surfaces with their tangent lines.

For surfaces in $\mathbb{R}P^3$, Platonova \cite{Platonova, Platonova2} completed the classification
of generic surfaces by projective transformations. The classification of generic two parameter families of surfaces
is done in \cite{DeolindoKabataOhmotoSano} (see also \cite{Arnoldbooklet,Goryunov, Kabata, Landis}).
The study of surfaces in $\R P^4$ was proposed in \cite{Arnoldency} with relation to the classification of singularities 
which appear in central projections of generic surfaces in $\R^4$ by D. Mond \cite{mond1, mond, mond2}. 
However there have been no results about the classification of surfaces so far.
This paper gives an answer to the proposal in \cite{Arnoldency} with the complete classification list of jets of generic surfaces in $\R P^4$  by projective transformations.

%$\A$-classification list of smooth map germs $\R^2,0\to\R^3,0$ of corank one by D. Mond.

On the other hand, the study of geometric aspects of surfaces in $4$-space is a relatively new subject and has a lot of analogy to the study of surfaces in $3$-space which have been investigated in \cite{bruce-nogueira, r4surf, little, dmm,projsrfR4}.

Let $M$ be a smooth surface embedded in $\mathbb{R}^4 \subset \mathbb{R} P^4$ containing the origin of $\R^4$
 where $\mathbb{R}^4$ is identified with the open subset $\{ [x;y;z;w;1]\}\subset \mathbb{R} P^4$. 
 We write  $(z,w)=f(x,y)=(f_1(x,y),f_2(x,y))$ as the Monge form of $M$ at the origin where $f_i(0,0)=df_i(0,0)=0$ for $i=1,2$. 
Two jets of surfaces at some points are said to be \emph{projective equivalent}
if there is a projective transformation on $\mathbb{R}P^4$ sending one to the other.
Our result is the following.
\begin{theorem}\label{mainthm} 
%Consider a smooth surface $M$ in $\mathbb{R}P^4$, and 
There is an open everywhere dense subset $\mathcal{O}$ of the space of compact smooth surfaces $M$ in $\mathbb{R}P^4$ 
such that the germ at each point on $M$ in $\mathcal{O}$ is projectively equivalent to a germ with the $k$-jet of the Monge form of one of the cases in Table  
\ref{bigstrata1}.
\end{theorem}

{%\footnotesize
\begin{table}[h]
\begin{tabular}{c|c|c|c|c}
\hline
 Type &  Normal form & Condition &  cod & Proj.\\
%\hline
%  $(x^2,y^2)$      & $-$ & $2$ & $0$ \\
%  $(xy,x^2-y^2)$   & $-$ & $2$ & $0$ \\
%  $(x^2,xy)$       & $-$ & $2$ & $1$ \\
%  $(x^2\pm y^2,0)$ & $-$ & $2$ & $2$ \\
  \hline
  \hline
  $\Pi_{E}$        &  $(x^2-y^2+y^2(\phi_1+\phi_2),xy+\psi_4)$ & $$          & $0$ & $-$   \\
  $\Pi_{S}$        &  $(x^2+y^3+ y\phi_3,y^2+\alpha x^3 + x\psi_3)$ & $\alpha\neq0$         & $0$ &$S$   \\
  %\hline
  $\Pi_{B}$     &$(x^2+y^3+ y\phi_3,y^2 + x\psi_3)$  & $-$                & $1$ & $S, B$ \\
  %\hline
  $\Pi_{2B}$   & $(x^2+ y\phi_3,y^2 + x\psi_3)$        & $-$                  & $2$ & $B$\\
 % \hline
 $\Pi_{H}$      &      $(x^2+\beta xy^2+y^3+y\phi_3,xy+x\psi_3)$              & $-$               & $1$  & $H$\\
 %                  \hline
 $\Pi_{P}$      &   $(x^2+xy^2+\lambda y^4,xy+\gamma y^3+\psi_4)$ & $\gamma,\Lambda\neq0$ & $2$ &$P$\\
%\hline
$\Pi^+_{I} $ & $(x^2+ y^2+k_1 x^2y+y\phi_3, \psi_3+ \psi_4)$ & $b_{30}-b_{12}\neq0$ &$2$ & $S,B$ \\
$\Pi^-_{I} $ & $(xy+k_2 x^3+\phi_4, \psi_3+ \psi_4)$ & $b_{30}, b_{03}\neq0, a_{22}=0$   &$2$ & $S,B,H$\\
\hline
  \end{tabular}
\caption {{Strata of codimension $ \leq 2$ in the space of
$4$-jets of Monge forms corresponding to $\A^3$-types of central projections on asymptotic lines.  
Here $\phi_s=\sum_{i+j=s}a_{ij}x^iy^j, \psi_s=\sum_{i+j=s}b_{ij}x^iy^j$,
$\alpha, \beta, \gamma, \lambda,  k_1, k_2, a_{ij}, b_{ij}\in\R$ are moduli parameters and $\Lambda=6\gamma^2+4\lambda-15\gamma+5$. Surface germs of $\Pi_E$-type do not have asymptotic lines.}
%and $\bar{\xi}_4$ is a homogeneous polynomial of degree $4$ without the term $x^2y^2$
% The last column describes the types of central projections of each surface (see Section 2 especially for $\dagger$ and $\dagger\dagger$).
%Refer to Section 4 about singularities of $\dagger$ and $\dagger\dagger$.
}
\label{bigstrata1}
\end{table}

Observe that the last column of Table \ref{bigstrata1} means the types of central projections of corresponding surfaces from view points on asymptotic lines. %
In section 2, we briefly explain the stratification of the space of jets of Monge forms 
induced by the stratification of jets of germs of central projections,
and review the stratification of the $3$-jet space of Monge forms induced from the $\A^3$-stratification of central projection germs
which was originally done in Ph. D thesis of Mond \cite{mond1}.
In Section 3 we obtain simple normal forms of jets of Monge forms that represent each stratum given in the previous section 
by projective transformations, and give the proof of Theorem \ref{mainthm}.

\begin{remark}
Our normal forms in Table \ref{bigstrata1} contain a lot of moduli parameters including coefficients of
higher order terms of degree grater than $4$. 
They must be interpreted as some projective differential invariants.
For example, when we look at the $\A$-types of the central projection of the $\Pi_S$-type surfaces germs,
it is observed that  the central projection from a view point on the asymptotic line is
$\A$-equivalent to $P_3(c): (x,xy^2+cy^4,xy+y^3)$ where $c$ is a moduli parameter, and $\frac{\lambda}{\gamma}=c$. The first author \cite{Deolindo} found also that $\gamma$ and $\lambda$
are expressed by combinations of some cross-ratio invariants and they determines the topological type of BDE (binary differential equations) of asymptotic curves.
\end{remark}

\
%\begin{center}
%\includegraphics[width=5cm, clip]{CurvasCentralProjection22.pdf}
%\end{center}

\textit{Acknowledgements:} We would like to thank Takashi Nishimura and Farid Tari for organizing the JSPS-CAPES no.002/14 bilateral project in 2014-2016.  The authors are supported by the project for their stays in ICMC-USP and Hokkaido University, respectively. The first author thanks also the FAPESP no.2012/ 00066-9 to support part of this work. We are also very grateful to Farid Tari and Toru Ohmoto for their supervisions.

\section{The stratification of the $3$-jet space of Monge forms}
Mond, in Chapter III of his Ph.D thesis \cite{mond1}, stratified the jet space of Monge forms according to $\A$-types of germs of central projections.
In this section we explain the way to stratify the jet space of Monge forms according to types of central projections, 
and review Mond's stratification for the $3$-jet space of Monge forms.

%In this section we review Mond's stratification of the $3$-jet space of Monge forms 
%in Chapter III of his Ph. D thesis \cite{mond1}.

%according to the types of germs of central projections 
%in Chapter III of his Ph. D thesis \cite{mond1}
%where he deals with central projections of surfaces in $\R^4$.

%In this section we stratify the jet space of Monge forms according to the types of germs of central projections.
%Note that similar results were given by Mond in Chapter III of his Ph. D thesis \cite{mond1}where he deals with central projections of surfaces in $\R^4$.

Take a surface $M$ in $\R^4\subset\R P^4$ with the Monge form $(z,w)=f(x,y)=(f_1(x,y),f_2(x,y))$.
 Let $V_{\ell}$ denote the space of polynomials in $x,y$ of degree greater than $1$ and less than or equal to $\ell$.  
Our aim here is to obtain a stratification of the $\ell$-jet space of Monge forms $V_{\ell} \times V_{\ell}$
which is induced from the $\A^\ell$-stratification of $J^\ell(2,3)$ obtained in \cite{mond1,mond} as follows.

%Here $\A^\ell$ means the group of $\ell$-jets of elements in $\A$ which acts on the $\ell$-jet space of map germs.
%Remark that our discussion for $\A$-equivalence is deduced to that for $\A^\ell$-equivalence,
%since all germs dealt with in the present paper are finitely $\A$-determined \cite{mond}.

Consider a point $p \in \mathbb{R} P^4$, which is sometimes called {\it a view point}, not lying on $M$ and define $\pi_p: \mathbb{R} P^4 -\{p\}\rightarrow \mathbb{R} P^3$ 
as the canonical projection which associates $x\in\R P^4-\{p\}$ to the line generated by $x-p$. 
{\it The central projection of the surface $M$ from $p$} is given by the composite map 
$$\varphi_{p,M}:=\pi_p \circ \iota: M \rightarrow \mathbb{R} P^3$$
(see also \cite{DeolindoKabataOhmotoSano}).

%\textcolor{red}{(more precisely?)}.

%There is 4-dimensional freedom of the choice of viewpoint $p$, there is naturally produced a
%4-parameter family of central projection, $M \times U \to \mathbb{R}P^3,$ where $U$ is any small open subset of
%the complement $\mathbb{R}^4 - M$.
%Therefore we might have expected that any germ in $\mathbb{R}^2\to\mathbb{R}^3$ of
%$A_e$-cod $\leq 4$  would appear in central projection generically. But it is not the case.

\begin{table}\label{gibsontable}
$$
\begin{array}{c|c | c }
\hline
 \mbox{Type}&\mbox{Normal form}  &\mbox{cod}\\
\hline
\hline
\mbox{hyperbolic}&(x^2,y^2)  &0\\
\mbox{elliptic}&(x^2-y^2,xy) &0\\
 \mbox{parabolic}&(x^2,xy) &1\\
 \mbox{inflection}&(x^2+ y^2,0)\;\;\mbox{or}\;\;(xy,0)&2\\
 \mbox{degenerate inflection}&(x^2,0)&3\\
 \mbox{degenerate inflection}&(0,0)&4\\
 \hline
\end{array}
$$
\caption{The classification of $J^2(2,2)$ (which is equal to the $2$-jet space of Monge forms $f=(f_1,f_2)$)  by $GL(2)\times GL(2)$-actions given by Gibson in \cite{gibson}.}\label{gibson}
\end{table}

We denote the central projection of the surface germ expressed in Monge form $f$ from a view point $p$ by $\varphi_{p,f}$.
We stratify $V_\ell\times V_\ell$ by the difference of $\A^\ell$-types of $j^\ell \varphi_{p,f}$ for a view point $p\in\R P^4-M$ and the $\ell$-jet of a Monge form $j^\ell f\in V_\ell\times V_\ell$.
$\A^\ell$ means the equivalence of jets of map germs, i.e., two jets $j^\ell g, j^\ell h\in J^\ell (2,3)$ are equivalent 
if and only if there exist jets of diffeomorphism germs $\sigma, \tau$ of the source and the target at the origins
such that $j^\ell h=j^\ell(\tau\circ g \circ \sigma^{-1})$. 
Remark that even for the same $\ell$-jet of the surface the central projection gives different $\A^\ell$-types depending on the view point.
For example, $j^\ell \varphi_{p,f}$ is always regular type (i.e. equivalent to $(x,y,0)$) if and only if $p$ is outside the tangent plane of the surface germ;
and gives a singularity if and only if $p$ is on the tangent plane to the surface.

We say that a line on the tangent plane which goes through the origin in $\R^4$ is an {\it asymptotic line of the surface given in Monge form $f$}  
if $\varphi_{p,f}$ is equivalent to one of singularities worse than a crosscap ($S_0$-type) for all view points $p$ on the line.   
%$S_0$-type (i.e. equivalent to $(x,xy,y^2)$) if and only if $p$ is outside the tangent plane of the surface germ 
 In Section 3 we show that the classification of $2$-jets of Monge forms by projective transformations coincides with the classification of $J^2(2,2)$ by $GL(2)\times GL(2)$-actions
 in Table \ref{gibson},
and each orbit is characterized by the number of asymptotic lines \cite{bruce-nogueira,mond}.
 For elliptic type with the form $j^2f=(x^2+y^2,xy)$, there are no asymptotic lines;
 for hyperbolic type with the form $j^2f=(x^2,y^2)$, $x$ and $y$-axis are two unique asymptotic lines;
 for parabolic type with the form $j^2f=(x^2,xy)$, $y$-axis is a unique asymptotic line;
 for inflection type with the form $j^2f=(x^2+y^2,0)$ or $(xy,0)$, 
 all lines on the tangent plane which go through the origin are asymptotic lines.
Thus we study types of central projections from view points on the asymptotic lines, and divide strata in Table \ref{gibsontable} into finer ones.

%get the finer strata \textcolor{red}{(i.e. proper subset)} of the stratification given in Table \ref{gibsontable}.  

 %the ways for singularities to appear in central projections of each type of Monge form is totally different:
 %For elliptic type with the form $j^2f=(x^2+y^2,xy)$, $j^2\varphi_{p,f}$ gives only $S_0$-type for all $p$ on the tangent plane;
 %For hyperbolic type with the form $j^2f=(x^2,y^2)$, $j^2\varphi_{p,f}$ is equivalent to $(x,y^2,0)$ for $p$ on $x$ and $y$-axis,
 %and to $S_0$-type for $p$ on the tangent plane but outside the above axes;
 %For parabolic type with the form $j^2f=(xy,x^2)$, $j^2\varphi_{p,f}$ is equivalent to $(x,xy,0)$ for $p$ on $y$-axis,
 %and to $S_0$-type for $p$ on the tangent plane but outside the above axis;
 %For inflection type with the form $j^2f=(xy^2+y^2,0)$ or $(xy,0)$, $j^2\varphi_{p,f}$ is equivalent to both $(x,y^2,0)$ and $(x,xy,0)$ for $p$ on some special lines (see Section 4 for the detail), and to $(x,y^2,0)$ for $p$ on the tangent plane but outside the above axes.

% central projections gives different types for each orbit of 2-jet of surfaces represented by normal forms in Table \ref{gibsontable}; for elliptic surfaces with the form $j^2f=(x^2+y^2,xy)$ $j^2\varphi_{p,f}$ gives 
 
 %\textcolor{red}{ The stratification for generic surfaces is enough if the codimension stratum is at most 2,}
% For generic surfaces the stratification is enough with the stratum of codimension at most $2$},
% which is verified by the following argument. 
 Define a smooth map, the Monge-Taylor map $\Theta:M  \rightarrow  V_{\ell}\times V_{\ell}$, 
 which associates to each point in $M$ the $\ell$-jet of Monge form $f=(f_1,f_2)$ at the point. The following is a natural extension of Bruce's Theorem in \cite{Bruce}.

\begin{theorem}\label{brucethm}  Let $Z \subset V_{\ell}\times V_{\ell}$ be an $GL(2)\times GL(2)$-invariant submanifold. For generic surface $M$ in $\mathbb{R}^4$, the Monge-Taylor map $\Theta: M \rightarrow V_{\ell}\times V_{\ell}$ is transverse to $Z$.
\end{theorem}
\begin{proof} The proof  follows the same arguments in the proof of Theorem 1 in \cite{Bruce} (see also \cite{r4surf}).\qed
\end{proof}

Since the strata are induced from $\A^\ell$-types of central projections,
they are $GL(2)\times GL(2)$-invariant. By Theorem \ref{brucethm} we consider only strata with codimension at most $2$. 
%Thus the strata of degenerate inflection-types in Table \ref{gibsontable} are out of our consideration,
%In addition, we do not care of dividing the elliptic-stratum,
%because the surface germ of the elliptic type does not have asymptotic lines.
%in order to study generic surfaces from singularity theory via central projections,
%we should get the defining condition of $G_W$ for corresponding $W\subset J^\ell(2,3)$
%whose codimension is at most 2.
%and for inflection-types we only consider open proper subsets of them as finer strata since inflection-types already have codimension $2$.
In this section, we study the stratification of 3-jets of Monge forms which is
induced by the $\A^3$-orbits in \cite{mond1, mond, mond2} or their unions given in Table \ref{MondStrata}.

\begin{table}
$$
\begin{array}{c|c  }
\hline
 \mbox{Name}&\mbox{Normal form} \\
\hline
\hline
S_0 & (x,y^2,xy)  \\
S      & (x,y^2,y^3+x^2y) \mbox{ or $(x,y^2,y^3)$ }\\
B      &(x,y^2,x^2y) \mbox{ or $(x,y^2,0)$} \\
H      &(x,xy,y^3) \\
P     & (x,xy+y^3,xy^2) \\
R    & (x,xy,xy^2)\\
T     & (x,xy+y^3,0)\\
U     & (x,xy,0)\\
 \hline
\end{array}
$$
\caption{$\A^3$-orbits of germs $\R^2,0\to\R^3,0$ \cite{mond1, mond}.}\label{MondStrata}
\end{table}

%$S_0: (x,y^2,xy)$, $S: (x,y^2,y^3+x^2y)$ or $(x,y^2,y^3)$, $B: (x,y^2,x^2y)$ or $(x,y^2,0)$, \\
%$H: (x,xy,y^3)$, $P: (x,xy+y^3,xy^2)$, $R: (x,xy,xy^2)$, 
%$T: (x,xy+y^3,0)$.
%, $X: (x,y^3,x^2y+xy^2)$, $Y: (x,y^3-x^2y,xy^2)$.

%Observe that the finer stratification (including that of the inflection-type) by $\A$-orbits is obtained in Section 4 after getting good normal forms of 3-jets of surfaces via projective transformation in Section 3.

%In addition, here we do not consider stratum of elliptic type since the Monge form with elliptic type gives just cross-cap type, and of stratum of inflection type since this stratum already has codimension $2$ (The open proper strata of this type are studied in detail in Section 4).

Write
$$
f_1(x,y)=\sum_{i+j\geq2}a_{ij}x^{i}y^i,\;\;   f_2(x,y)=\sum_{i+j\geq2}b_{ij}x^{i}y^i.
$$
% and by the same discussion in subsection 2.1 of \cite{DeolindoKabataOhmotoSano}, 
Then $\varphi_{p,f}$ can be regarded as a map germ $\R^2,0\to\R^3$. Indeed, for $p=[a;b;c;d;1] \in \R^4-M$, 
%with $a^2+b^2+c^2+d^2\neq0$
we choose $a\neq0$, then
$\varphi_{p,f}$ is given by 
$$
\varphi_{p,f}(x,y)=\left(\frac{y-b}{x-a},\frac{f_1(x,y)-c}{x-a},\frac{f_2(x,y)-d}{x-a}\right).
$$
On the other hand, if $p$ is taken at infinity and written as $p=[a;b;c;d;0]$,   
$\varphi_{p,f}$ is given by
$$
\varphi_{p,f}(x,y)=(y-ux,f_1(x,y)-vx,f_2(x,y)-wx)
$$
with $(u,v,w)=(\frac ba,\frac ca,\frac da)$ (see also \cite{DeolindoKabataOhmotoSano}).

The following sums up the Propositions III. 2:2, 2:8, 2:14, 2:16 and 2:17 in \cite{mond1}. 
\begin{proposition}{\rm (Mond \cite{mond1})} \label{stratification}
 For a surface germ in Monge form $(z,w)=f(x,y)$ we have the following.
\begin{itemize}
%\item[(i)] elliptic type,  $\varphi_{p,f}(0)\sim_{\A} S_0$ for any $p\in T_qM$.
\item[(i)]   Suppose that $j^2f=(x^2,y^2)$. Then: %hyperbolic type with $j^3f=(x^2+\sum_{i+j=3}a_{ij}x^iy^j,y^2+\sum_{i+j=3}b_{ij}x^iy^j)$ for which $x$-axis and $y$-axis are asymptotic directions,
$$
\begin{array}{ccr}
j^3\varphi_{p,f}\sim S &\Leftrightarrow & \mbox{$a_{03}\neq0$ {\rm (}resp. $b_{30} \neq0${\rm )}}\\
j^3\varphi_{p,f}\sim B &\Leftrightarrow& \mbox{$a_{03}=0$ {\rm (}resp. $b_{30} =0$\rm{)}}
\end{array}
$$
for $p$ on the asymptotic line $x=0$ (resp. $y=0$).
\item[(ii)]  Suppose that $j^2f=(x^2,xy)$. Then: %parabolic type with $j^3f=(x^2+\sum_{i+j=3}a_{ij}x^iy^j,xy+\sum_{i+j=3}b_{ij}x^iy^j)$,
$$
\begin{array}{ccl}
j^3\varphi_{p,f}\sim H &  \Leftrightarrow & \mbox{ $a_{03}\neq0$ }\\
j^3\varphi_{p,f}\sim P &  \Leftrightarrow &\mbox{ $a_{03}=0$ and $a_{12},b_{03}\neq0$}\\
j^3\varphi_{p,f}\sim R &  \Leftrightarrow &\mbox{ $a_{03}=b_{03}=0$ and $a_{12}\neq0$}\\
j^3\varphi_{p,f}\sim T &  \Leftrightarrow &\mbox{ $a_{03}=a_{12}=0$ and $b_{03}\neq0$}\\
j^3\varphi_{p,f}\sim U & \Leftrightarrow &\mbox{ $a_{03}=a_{12}=b_{03}=0$}\\
\end{array}
$$
for $p$ on the unique asymptotic line $x=0$.
\item[(iii)]  Suppose that $j^2f=(x^2+y^2,0)$. Then: %parabolic type with $j^3f=(x^2+\sum_{i+j=3}a_{ij}x^iy^j,xy+\sum_{i+j=3}b_{ij}x^iy^j)$,
$$
\begin{array}{ccl}
j^3\varphi_{p,f}\sim S\; \mbox{or}\; B 
\end{array}
$$
for any $p$ on the $xy$-plane.
\item[(iv)]  Suppose that $j^2f=(xy,0)$. Then: %parabolic type with $j^3f=(x^2+\sum_{i+j=3}a_{ij}x^iy^j,xy+\sum_{i+j=3}b_{ij}x^iy^j)$,
$$
\begin{array}{ccl}
j^3\varphi_{p,f}\sim S, B\; \mbox{or}\; H&  \Leftrightarrow & \mbox{ $b_{30}\neq0$ }\\
j^3\varphi_{p,f}\sim S, B\; \mbox{or}\; P &  \Leftrightarrow &\mbox{ $b_{30}=0$ and $a_{30},b_{21}\neq0$}\\
j^3\varphi_{p,f}\sim S, B\; \mbox{or}\; R &  \Leftrightarrow &\mbox{ $a_{30}=b_{30}=0$ and $b_{21}\neq0$}\\
j^3\varphi_{p,f}\sim S, B\; \mbox{or}\; T &  \Leftrightarrow &\mbox{ $a_{30}=b_{21}=0$ and $b_{30}\neq0$}\\
j^3\varphi_{p,f}\sim S, B\; \mbox{or}\; U & \Leftrightarrow &\mbox{ $a_{30}=b_{30}=b_{21}=0$}\\
\end{array}
$$

for any $p$ on the $xy$-plane.

%\item[(iv)]  inflection type with $j^3f(0)=(x^2\pm y^2+\sum_{i+j=3}a_{ij}x^iy^j,\sum_{i+j=3}a_{ij}x^iy^j)$,
%$$
%\begin{array}{l}
%\textcolor{red}{j^3\varphi_{p,f}(0)\sim S_1, S, B \mbox{  or $H$  $\Leftrightarrow$ }}
%\end{array}
%$$
%for $p$ on an asymptotic directions.
\end{itemize}
\end{proposition}

\proof
Statement (i).
It is easy to check that the $x$ and $y$-axes are asymptotic lines at the origin for $M$ in Monge-form with $j^2f=(x^2,y^2)$.
Suppose $p$ is at the $y$-axis and written as $p=(0,a,0,0)$,
then, by coordinate changes, we get
$$
j^3 {\varphi}_{p,f}\sim_{\A^3}(x,a(a_{21}a-1)x^2y+a_{03}y^3,y^2).
$$
If $a_{03}\neq0$, $j^3\varphi_{p,f}(0)$ is of $S$-type, otherwise it is of $B$-type.
%More precisely, the value $a$ determines the finer strata, i.e., if $a_{21}a-1=0$, $j^3\bar{\varphi}_{p,f}(0)$ is $S$ or $C$,
%unless $j^3f(0)$ is $S_1$ or $B$. 
If $p$ is at infinity on the $y$-axis, we obtain
$$
j^3 {\varphi}_{p,f}\sim_{\A^3}(x,xy,a_{21}x^2y+a_{03}y^3).
$$
Again, if $a_{03}\neq0$, $j^3\varphi_{p,f}(0)$ is of $S$-type, otherwise it is of $B$-type. %where the terms $a_{03}$ determine the same type described as the above with the same condition.
By exchanging $x$ and $y$ (also $a_{ij}$ and $b_{ij}$), the case of $p$ on the $x$-axis is follows similarly.   

Statement (ii). Remark that $y$-axis is a unique asymptotic line at the origin when $j^2f(0)=(x^2,xy)$,
hence we put $p=(0,a,0,0)\in\R^4\subset \R P^4$ ($a\neq0$) and we get
$$
j^3 \varphi_{p,f}\sim_{\A^3}(x,a_{12}axy^2+a_{03}y^3,xy+b_{03}y^3).
$$
Statement (ii) for $p\in\R^4$ naturally follows, and the case $p$ at infinity is similar.

Statement (iii).
Let view point $p=(a,b,0,0)\in\R^4$ and $a\neq0$, then we get 
$$
j^3 \varphi_{p,f}\sim_{\A^3}(x,  y^2,\xi_1x^2y+\xi_2y^3)
$$
where $\xi_1$ and $\xi_2$ are homogeneous polynomials of degree $3$ with variables $a$ and $b$ whose coefficients consist of $a_{ij}$ and $b_{ij}$.
Statement (iii) for $p=(a,b,0,0)\in\R^4$ where $a\neq0$ naturally follows, and it is easily seen that the projection gives just $S_1$-type for view points $p=(0,b,0,0)\in\R^4$ where $b\neq0$.
The case $p$ at infinity is similar.

Statement (iv).
For view point $p=(a,b,0,0)\in\R^4$ where $a, b\neq0$, it is easily seen that 
$
j^3 \varphi_{p,f}\sim S\; \mbox{or}\; B
$
in the similar way to the above.
Put $p=(a,0,0,0)\in\R^4$ where $a\neq0$ and we get 
$$
j^3 \varphi_{p,f}\sim_{\A^3}(x, xy-\frac{a_{30}}{a}y^3,b_{21}xy^2-\frac{b_{30}}{a}y^3).
$$
Statement (iv) for $p\in\R^4$ naturally follows, and the case $p$ at infinity is similar.
\qed
%where the term $a_{03}$ determines the same type described as the above,
%and the degenerated type $S$ or $C$ happens when $a_{21}$=0.

Based on Proposition \ref{stratification}, we stratify the $3$-jet space of Monge forms into strata with codimension at most $2$ as in Table \ref{strattable}.
In the next section we give simple normal forms of $4$-jets of Monge forms which represent each stratum in Table \ref{strattable} by projective transformations.

\begin{table}[h]
\begin{tabular}{c|c|c|c|c}
\hline
 Name &  Type of $2$-jet & Condition & cod &Proj.\\
%\hline
%  $(x^2,y^2)$      & $-$ & $2$ & $0$ \\
%  $(xy,x^2-y^2)$   & $-$ & $2$ & $0$ \\
%  $(x^2,xy)$       & $-$ & $2$ & $1$ \\
%  $(x^2\pm y^2,0)$ & $-$ & $2$ & $2$ \\
  \hline
  \hline
  $\Pi_{E}$        &  $(x^2-y^2,xy)$ & $-$         &  $0$ &$-$   \\
 \hline
  $\Pi_{S}$        & $(x^2,y^2)$& $a_{03}\cdot b_{30}\neq0$         &  $0$ &$S$   \\
  %\hline
  $\Pi_{B}$     & & $b_{30}=0,\; a_{03}\neq0$                &  $1$ & $S, B$ \\
  %\hline
  $\Pi_{2B}$   &        & $a_{03}=b_{30}=0$                 &  $2$ & $B$\\
  \hline
 $\Pi_{H}$      &   $(x^2,xy)$      & $a_{03}\neq0$               &$1$  & $H$\\
 %                  \hline
 $\Pi_{P}$      &    & $a_{03}=0,\; a_{12}\cdot b_{03}\neq0$ & $2$ &$P$\\
\hline
$\Pi^+_{I} $ & $(x^2+y^2,0)$ & $-$ &$2$ & $S,B$ \\
$\Pi^-_{I} $ & $(xy,0)$ & $b_{30}\neq0$   &$2$ & $S, B, H$\\
\hline
  \end{tabular}
\caption {Strata of codimension $ \leq 2$ in the space of
$3$-jets of Monge forms corresponding to $\A^3$-types of central projections from view points on asymptotic lines. 
Surface germs of $\Pi_E$-type do not have asymptotic lines.  
 %The last column describes the types of central projections of each surface (see Section 2 especially for $\dagger$ and $\dagger\dagger$).
%Refer to Section 4 about singularities of $\dagger$ and $\dagger\dagger$.
}\label{strattable}
\end{table}
\begin{remark}{\rm
%\textcolor{red}{(The set $\A^k$ is the set of $k$-jets of elements of $\mathcal A$.)}
By taking higher order terms of the  simple normal forms in the Table \ref{bigstrata1},
we can consider a finer stratification of the space of Monge forms which corresponds to $\A$-types of central projections as Mond did in \cite{mond1}.
For instance, we take the Monge form of the $\Pi_{S}$-type in Table \ref{bigstrata1} and write
$f=(x^2+y^3+\sum_{i+j\ge4}a_{ij}x^iy^j,y^2+\alpha x^3+\sum_{i+j\ge4}b_{ij}x^iy^j)$ with $\alpha, a_{ij}, b_{ij} \in\R$, $\alpha\neq0$ and $a_{40}=b_{04}=0$.
Then the condition $a_{31}, b_{13}\neq0$ determines the proper stratum of the $\Pi_{S}$-stratum where 
the $\A$-types of the central projection can be determined as the regular, crosscap, $S_1$ or $S_2$-type 
depending on the position of the view points (See \cite{mond1}).}
\end{remark}

\section{The classification of Monge forms by projective transformations and proof of Theorem \ref{mainthm}}

%We can write a matrix $\mathcal{M}$ of action $G(5)$ like
%$$\mathcal{M}=\left(
%    \begin{array}{ccccc}
%      q_{11} & q_{12} & q_{13} & q_{14} & 0 \\
%      q_{21} & q_{22} & q_{23} & q_{24} & 0 \\
%      0 & 0 & q_{33} & q_{34} & 0 \\
%      0 & 0 & q_{43} & q_{44} & 0 \\
%      \alpha & \beta & \gamma & \delta & 1 \\
%    \end{array}
%  \right)=\left(
%    \begin{array}{ccc}
%     \mathcal{A}_{2\times 2}  & \mathcal{C}_{2\times 2}  &  0 \\
%     0_{2\times 2}  & \mathcal{B}_{2\times 2}  &  0 \\
%     \Upsilon_1 & \Upsilon_2 & 1 \\
%        \end{array}
%  \right)
%$$
%that represent the projective transformations

In this section we consider a classification of jets of Monge-forms of generic surfaces by projective transformations based on the stratification in Table \ref{strattable}.
 The projective linear group $PGL(5)$ is defined as the quotient space $GL(5)/\sim$, where $A\sim A' $
if  $\exists \lambda \in\mathbb{R}$  such that  $A=\lambda  A'$. 
%\textcolor{red}{Two surface germs in $\R^4\subset\R P^4$ are said to be projectively equivalent if there exists a projective transformation (an element of $PGL(5)$)  which maps one to the other.}
To consider the action on $V_{\ell}\times V_\ell$ ($\ell$-jet-spaces of Monge-forms),  
we define the following subgroup 
$$
G(5):=\{\Psi \in PGL(5) \;|\; \Psi(0)=0, \;\; \Psi(W)=W\}
$$
of $PGL(5)$, where $0=[0;0;0;0;1]$ is the origin and $W$ is the $xy$-plane in $\mathbb{R}^4$.
Thus $G(5)$ form a $16$-dimensional subgroup of $PGL(5)$ and acts on $V_{\ell}\times V_{\ell}$.

Let $f=(f_1,f_2)$ and $g=(g_1,g_2)$ be Monge forms of surface-germs at the origin.
We say that the $k$-jets of these Monge forms are {\it projectively equivalent} and write $j^kf\sim j^kg$
if there exists $\Psi\in G(5)$ which transforms one to the other.
Remark that $\A^\ell$-types of central projections of jets of smooth surfaces 
are invariant under projective transformations of surfaces,
that is,
$j^\ell \varphi_{p,f}\sim_{\A^\ell} j^\ell \varphi_{\Phi(p),\Phi(f)}$ from view point $p$ and $\Phi\in G(5)$.

In this paper we check the equivalence of jets of Monge forms in the following way.
With the coordinate $(x,y,z,w)$ of $\R^4$,
a projective transformation $\Psi\in G(5)$ is  regarded locally as a diffeomorphism germ $\R^4,0\to\R^4,0$ given by
$$
\Psi(x,y,z,w)=\left(\frac{q_1(x,y,z,w)}{p(x,y,z,w)},\frac{q_2(x,y,z,w)}{p(x,y,z,w)},\frac{q_3(x,y,z,w)}{p(x,y,z,w)},\frac{q_4(x,y,z,w)}{p(x,y,z,w)}\right),
$$ 
where $q_{i}=q_{i1}x+ q_{i2}y + q_{i3}z + q_{i4}w$, for $i=1,2$, $q_{j}= q_{j3}z + q_{j4}w$, for $j=3,4$ and $p= 1+p_1 x+ p_2 y+ p_3 z + p_4 w$.
 Define
$$\begin{array}{l}
F_1(x,y,z,w)=\frac{q_3}{p}-f_1(\frac{q_1}{p},\frac{q_2}{p})\\
F_2(x,y,z,w)=\frac{q_4}{p}-f_2(\frac{q_1}{p},\frac{q_2}{p}).
 \end{array}
$$
Then
$$
F_1(x,y,g_1,g_2)=F_2(x,y,g_1,g_2)=o(k)
$$
where $o$ is Landau's symbol implies $j^kf\sim j^kg$. 
%If we write
%$$
%f_1=\sum_{i+j\geq2}a_{ij}x^{i}y^i \mbox{ and $\displaystyle  f_2=\sum_{i+j\geq2}b_{ij}x^{i}y^i$},
%$$

Hence, to check the equivalence, we have to solve algebraic equations $F_1=F_2=o(k)$ in terms of $q_is$ and $p_is$ 
for a given Monge form $f=(f_1,f_2)$ and some simplified normal form $g=(g_1,g_2)$.
Recall that we already have a stratification of the $3$-jet space of Monge forms induced from the $\A^3$-stratification as in Table \ref{strattable},
hence our task  is to find a simple normal form of each stratum by projective equivalence. 
We begin with simplifying $2$-jets of Monge forms and then deal with higher jets.
However we stop this process with the $4$-jets,
since the dimension of $G(5)$ which act on the jet space of Monge forms is just $16$
and it does not give so good normal forms for higher jets.

%$a_{ij}$ and $b_{ij}$ for a given Monge form $f=(f_1,f_2)$ and some simplified normal form $g=(g_1,g_2)$.

%$j^2(f_1,f_2)=(\sum_{i+j=2}a_{ij}x^iy^j,\sum_{i+j=2}b_{ij}x^iy^j)$

\subsection{$2$-jet}
We first deal with the classification of $2$-jets of Monge-forms.
In the $2$-jet space, the condition $F_1=F_2=o(2)$ for any $j^2f, j^2g \in V_2\times V_2$ gives equations of just $q_{i1},q_{i2},q_{j3},$ $q_{j4}$ with $i=1,2$ and $j=3,4$,
and the classification by projective transformations is reduced to the classification of $V_2\times V_2\subset J^2(2,2)$ by the natural action of
$\mathcal{G}=GL(2,\mathbb{R})\times GL(2,\mathbb{R}).$
The $\mathcal{G}$-orbits are classified in \cite{gibson} described as Table \ref{gibson}.
%\textcolor{red}{Hence, the assumption of Section 2 is verified. } 
We classify now the higher jets of germs with a $2$-jet as in Table \ref{gibson}.

%Write $j^2(f_1,f_2)=(\sum_{i+j=2}a_{ij}x^iy^j,\sum_{i+j=2}b_{ij}x^iy^j)$ where $a_{ij}, b_{ij} \in\R$. Since the $2$-jet depends only on the coefficients $q_{i1},q_{i2},q_{j3},$ $q_{j4}$ with $i=1,2$ and $j=3,4$, the $G(5)$ became, for this case, the group regarded as $$\mathcal{G}=GL(2,\mathbb{R})\times GL(2,\mathbb{R})$$ that acts on the space $J^2(2,2)$. The $\mathcal{G}$-orbits are classified in \cite{gibson}, described as \textcolor{red}{Table \ref{gibson}.}
%Hence, the assumption of Section 2 is verified, and we classify the higher jets of germs with these $2$-jet.

%codimension in the Table \ref{gibson} .
\subsection{Elliptic case}
Suppose that $j^2f=(x^2-y^2,xy)$ and write 
$$
j^3(f_1,f_2)=(x^2-y^2+\sum_{i+j=3}a_{ij}x^iy^j,xy+\sum_{i+j=3}b_{ij}x^iy^j)
$$
 where $a_{ij}, b_{ij} \in\R$.
The following equivalence
$$
j^3(f_1,f_2)\sim(x^2-y^2+y^2\phi_1,xy),
$$
is given by projective transformation $\Phi$ with
$$
\begin{array}{c}
q_1=x+b_{03}z+(-a_{31}+b_{12}-b_{30})w,\;\; q_2= y-b_{30}z+(-b_{21}+b_{03}+a_{30})w,\\
q_3 = z,\;\; q_4 = w,\;\; p =1+(a_{30}+2b_{03})x+(2b_{12}-a_{21})y.
\end{array}    
$$
Here $\phi_k$ means homogeneous polynomials of degree $k$.
Consider 
$$
j^4(f_1,f_2)=(x^2-y^2+y^2\phi_1+\sum_{ i+j=4}c_{ij}x^iy^j,xy+\sum_{i+j=4}d_{ij}x^iy^j)
$$ 
where $c_{ij}, d_{ij} \in\R$, then
$$
j^4(f_1,f_2)\sim(x^2-y^2+y^2(\phi_1+\phi_2),xy+\phi_4),
$$ 
by $\Phi$ with
$
q_1=x,\;\;q_2=y,\;\;q_{3}=z,\;\;q_4=w,\;\;p=1+c_{40}z+c_{31}w.
$

\subsection{Hyperbolic case}
Suppose that $j^2f=(x^2,y^2)$ and write 
$$
j^3(f_1,f_2)=(x^2+\sum_{i+j=3}a_{ij}x^iy^j,y^2+\sum_{i+j=3}b_{ij}x^iy^j)
$$
 where $a_{ij}, b_{ij} \in\R$.
The following equivalence
$$
j^3(f_1,f_2)\sim(x^2+a_{03}y^3,y^2+b_{30}  x^3)
$$
is given by projective transformation $\Phi$ with
$$
\begin{array}{c}
q_1=x+\frac{1}{2}(-a_{30}+b_{12})z -\frac{1}{2}a_{12}w, \;\; q_2=y -\frac{1}{2}b_{21}z+\frac{1}{2}(a_{21}-b_{03})w,\\
       q_{3}=z, \;\;q_4=w,\;\; p=1+ b_{12}x +a_{21}y.
\end{array}       
$$

%\begin{remark}\label{hyperem}  \emph{When we consider the finer stratification of the jet space of Monge forms induced by $\A$-classification, especially for strata in $\Pi_B$ and $\Pi_{2B}$, it is better to get the normal form for $5$-jet.
%\begin{remark}{\normalfont
 We can eliminate more two coefficients in $4$-jet.
Put 
$$
j^4(f_1,f_2)=(x^2+a_{03}y^3+\sum_{ i+j=4}c_{ij}x^iy^j,y^2+b_{30}  x^3+\sum_{i+j=4}d_{ij}x^iy^j)
$$ 
where $c_{ij}, d_{ij} \in\R$, then
$$
j^4(f_1,f_2)\sim(x^2+a_{03}y^3+ y\phi_3,y^2+b_{30}x^3 + x\psi_3),
$$ 
by $\Phi$ with
$
q_1=x,\;\;q_2=y,\;\;q_{3}=z,\;\;q_4=w,\;\;p=1-c_{40}z -d_{04}w.
$
Here $\phi_3$ and $\psi_3$ means homogeneous polynomials of degree $3$.
Then
$$
j^4(f_1,f_2)\sim\left\{\begin{array}{ll}
(x^2+y^3+y\phi_3, y^2+\alpha x^3+x\psi_3), \;\alpha\in\R^* & \mbox{if $a_{03}, b_{30}\neq0$;}\\
(x^2+y^3+y\phi_3, y^2+x\psi_3)  & \mbox{if $a_{03}\neq0$ and  $b_{30}=0$;}\\
(x^2+y\phi_3, y^2+x\psi_3)  & \mbox{if $a_{03}=b_{30}=0$.}\\
\end{array}\right.
$$

\subsection{Parabolic case}
Suppose that $j^2f=(x^2,xy)$ and write 
$$
j^3(f_1,f_2)=(x^2+\sum_{i+j=3}a_{ij}x^iy^j,xy+\sum_{i+j=3}b_{ij}x^iy^j)
$$ 
where $a_{ij}, b_{ij} \in\R$.
It is easy to show that
$$
j^3(f_1,f_2)\sim(x^2+a_{12}xy^2 +a_{03}y^3, xy+\bar{b}_{12}xy^2+b_{03}y^3)
$$
where $\bar{b}_{12}=b_{12}-\frac12 a_{21}$.
If $a_{03}\neq0$, then
$$
j^3(f_1,f_2)\sim(x^2+(a_{12}+3b_{03})xy^2+a_{03}y^3,xy)
$$
with the equivalence given by $\Psi$ with 
$$
\begin{array}{c}
q_1=x-\frac{(-\bar{b}_{12}a_{03}+3a_{12}b_{03}+3b_{03}^2)}{a_{03}}w,\;\\
q_2=\frac{b_{03}}{a_{03}}x+ y+ \frac{b_{03}^2(a_{12}b_{03}-a_{03}\bar{b}_{12})}{a_{03}^3}z -\frac{b_{03}(2b_{03}^2+\bar{b}_{12}a_{03})}{a_{03}^2}w,\;\;q_{3}=z,\\
q_4=\frac{b_{03}}{a_{03}}z+w,\;\;p=1 +\frac{b_{03}^2(a_{12}+b_{03})}{a_{03}^2} x-\frac{(-2\bar{b}_{12}a_{03}+4a_{12}b_{03}+3b_{03}^2)}{a_{03}} y.
\end{array}
$$
Then the $4$-jet can be written in the form 
$$
j^4(f_1,f_2)\sim(x^2+\beta xy^2+ y^3+y\phi_3,xy+x\psi_3),
$$
where $\beta=\frac{(a_{12}+3b_{03})}{a_{03}^{3/2}}$, $\phi_3$ and $\psi_3$ mean homogeneous polynomials of degree $3$.
%If $d=0$ we have by another projective transformation $$j^5(f_1,f_2)\sim(x^2+x^2\phi_1+y^3+x\psi_3+x\psi_4,xy+\phi_3+x\phi'_3+x\phi_3)$$

If $a_{03}=0$ but $a_{12}\neq0$ we obtain
$$
j^3(f_1,f_2)\sim(x^2+xy^2,xy+\gamma  y^3)
$$
with the projective transformation $\Psi$  given by
$$
\begin{array}{c}
q_1=a_{12}x +a_{12}b_{12}w, \;\;q_2=y,\;\; q_{3}=a_{12}^2z, \\
q_4=a_{12}w,\;\; p=1+2b_{12}y, 
\end{array}
$$
where $\gamma=\frac{b_{03}}{a_{12}}$.
If we put 
$$
j^4(f_1,f_2)=(x^2+xy^2+\sum_{i+j=4}a_{ij}x^iy^j,xy+\beta  y^3+\sum_{i+j=4}b_{ij}x^iy^j),
$$
then $\gamma\neq0$ leads to  
 $$
 j^4(f_1,f_2)\sim(x^2+xy^2+\lambda y^4,xy+\gamma  y^3+\phi_4)
 $$
 by a projective transformation $\Psi$ with 
  $$
  \begin{array}{c}
  q_1=x+\frac{1}{2}(-q_{21}^2+p_1)z+(3\gamma q_{21}-3q_{21})w,\\ 
  q_2=y+q_{21}x+\frac{1}{2}(-2\gamma q_{21}^3+q_{21}^3+p_1 q_{21})z+\frac{1}{2}(-q_{21}^2+p_1)w,\\ 
  q_{3}=z, \;\;q_4=q_{21}z+w,\;\;
  p=1+p_1 x - (-6\gamma q_{21}-4q_{21})y+p_3 z+p_4 w,
  \end{array}
  $$
   where $p_1=\frac{1}{\Lambda^2}\xi_1$, $p_3=\frac{1}{\Lambda^4}\xi_2$, $p_4=\frac{1}{\Lambda^3}\xi_3$, $q_{21}=-\frac{a_{13}}{\Lambda}$, $\xi_i$ are combinations of the coefficients of 4-jet and $\Lambda=6\gamma^2+4\lambda-15\gamma+5\neq0$. If $\Lambda=0$ the terms of $j^4(f^1,f^2)$ of order $4$ can not be removed.  The $\phi_4$  is a homogeneous polynomials of degree $4$.

\subsection{Inflection case}
Suppose that $j^2f=(x^2+y^2,0)$ and write 
$$
j^3(f_1,f_2)=(x^2+y^2+\sum_{i+j=3}a_{ij}x^iy^j,\sum_{i+j=3}b_{ij}x^iy^j).
$$
Let $b_{30}-b_{12}\neq0$. It follows that
$$
j^3(f_1,f_2)\sim(x^2+y^2+k_1x^2y,\phi_3)
$$
 by $\Psi$ with 
 $$
 \begin{array}{c}
 q_1=x, \;\;q_2=y, \;\;q_{3}=z+(a_{30}-a_{12})w, \;\;q_4=(b_{30}-b_{12})w,\\ 
 p=1-\frac{(a_{30}b_{12}-a_{12}b_{30})}{(b_{30}-b_{12})}x-\frac{(a_{30}b_{03}-a_{12}b_{03}-a_{03}b_{30}+a_{03}b_{12})}{(b_{30}-b_{12})}y. 
\end{array}
$$
Here, $k_1$ is scalar  constant. Now, we take
 $$
j^4(f_1,f_2)=(x^2+y^2+k_1x^2y+\sum_{i+j=4}c_{ij}x^iy^j,\phi_3+\sum_{i+j=4}d_{ij}x^iy^j),
$$
where $c_{ij},d_{ij} \in \R$, then it follows that
$$
j^4(f_1,f_2)\sim(x^2+y^2+k_1x^2y+y\psi_3,\phi_3+\phi_4)
$$
 by $\Psi$ with $q_1=x, \;\;q_2=y, \;\;q_{3}=z, \;\;q_4=w, p=1+c_{40}z.$ Here $\phi_k$ and $\psi_k$ means homogeneous polynomials of degree $k$.

 Next, suppose 
 $$
 j^3(f_1,f_2)=(xy+\sum_{i+j=3}a_{ij}x^iy^j,\sum_{i+j=3}b_{ij}x^iy^j).
 $$
If $b_{03}\neq0$, then
$$
j^3(f_1,f_2)\sim(xy+k_2x^3,\phi_3)
$$
 by $\Psi$ with 
 $$
 \begin{array}{c}
 q_1=x, \;\;q_2=y, \;\;q_{3}=z+(a_{21}+a_{03})w, \;\;q_4=b_{03}w,\\ 
 p=1-\frac{(a_{21}b_{03}-a_{03}b_{21})}{b_{03}}x-\frac{(a_{21}b_{03}-a_{03}b_{12})}{b_{03}}y.
\end{array}
$$
 The $k_2$ is a scalar constants. Finally, we consider
 $$
j^4(f_1,f_2)=(xy+k_2x^3+\sum_{i+j=4}c_{ij}x^iy^j,\phi_3+\sum_{i+j=4}d_{ij}x^iy^j),
$$
where $c_{ij},d_{ij} \in \R$.  Thus, it follows that
$$
j^4(f_1,f_2)\sim(xy+k_2x^3+\bar{\xi}_4,\phi_3+\phi_4)
$$
 by $\Psi$ with $q_1=x,\;q_2=y,\;q_{3}=z,\;q_4=w, \;p=1+c_{22}z.$
The $\phi_k$ means homogeneous polynomials of degree $k$ and $\bar{\xi}_4$ is a homogeneous polynomials of degree $4$ without the term $x^2y^2$. %These normal forms complete the classification of Table \ref{bigstrata1}.


\begin{thebibliography}{99999}
%\newcommand{\bysame}{%
%  \leavevmode\hbox to 3em{\hrulefill}\,}
\bibitem{Arnoldbooklet} V. I. Arnold,
{\it Catastrophe Theory}, 3rd edition, Springer (2004).
%
\bibitem{Arnoldency} V. I. Arnold, V. V. Goryunov, O. V. Lyashko, V. A. Vasil'ev,
{\it Singularity Theory II, Classification and Applications},
Encyclopaedia of Mathematical Sciences Vol. 39, Dynamical System VIII (V. I. Arnold (ed.)),
(translation from Russian version), Springer-Verlag  (1993).
%
%\bibitem{BT-Du} J. W. Bruce and F. Tari, Dupin indicatrices and families of curve congruences. Trans. Amer. Math. Soc. 357 (2005), 267-\UTF{00EF}\UTF{0153}\UTF{00C2}°85.
%
%\bibitem{BT1}   J. W. Bruce and F. Tari, On binary differential equations, Nonlinearity 8 (1995), 255--271.
%
%\bibitem{BT2}  J. W. Bruce  and F. Tari, Generic  $1$-parameter families of binary differential equations of Morse type, Discrete and continuous dynamical systems, 3 (1997), 79--90.
%
\bibitem{Bruce} J. W. Bruce, Projections and reflections of generic surfaces in $\mathbb R^3$.
{\it Math. Scand.} 54 (1984), 262--278.

\bibitem{BGT} J. W. Bruce, P. J. Giblin and F. Tari. Families of surfaces: height function, Gauss maps and duals. In 
{\it Real and Complex singularities (S\~ao Carlos, 1994)}, 148--178. Pitman Res. Notes Math. Ser. 333, Longman, Harlow, 1995.

\bibitem{bruce-nogueira} J. W. Bruce and A. C. Nogueira, Surfaces in ${{\mathbb R}}^4$ and duality.
{\it Quart. J. Math. Oxford Ser.} Ser. (2), 49 (1998), 433--443.

\bibitem{r4surf}  J. W. Bruce and F. Tari, Families of surfaces in ${\mathbb R}^4$.  {\it
Proc. Edinb. Math. Soc.} 45 (2002), 181--203.
%\bibitem{BFT}  J. W. Bruce, G. J. Fletcher  and F. Tari, Bifurcations of implicit differential equations, Proc. Royal Soc. Edinburgh, 130 A (2000), 485--506.
%
%\bibitem{CBR} M. Cibrario,
%Sulla reduzione a forma delle equationi lineari alle derviate parziale di secondo ordine di tipo misto,
%Accademia di Scienze e Lettere, Instituto Lombardo Redicconti, 65, 889-906 (1932).
%
%\bibitem{DR} L. Dara, Singularit\'{e}s g\'{e}n\'eriques des \'equations differentielles multiformes,
%Bol. Soc. Brasil Math. 6 (1975), 95--128.
%
%\bibitem{Davydov}  A. A. Davydov, Normal form of a differential equation,
%not solvable for the derivative, in a neighborhood of a singular point,
%Funct. Anal. Appl. 19 (1985), 1--10.
%
%\bibitem{Davydov2}  A. A. Davydov and E. Rosales-Gonsales,
%Smooth normal forms of folded elementary singular points,
%Jour. Dyn. Control Systems 1 (1995), 463--482.
%

\bibitem{Deolindo} J. L. Deolindo-Silva. Cr-invariants for surfaces in $\mathbb R^4$. Preprint, 2015.

\bibitem{DeolindoKabataOhmotoSano}  J. L. Deolindo-Silva, Y. Kabata, T. Ohmoto, H. Sano. Projective classification of jets of surfaces in 3-space, preprint, (2015).
%Projective classification of jets of surfaces in $\Proj^4$ and applications,
%preprint (2015).
%
\bibitem{gibson} C. G. Gibson, Singular points of smooth mappings, Pitman Research Notes in Mathematics, vol. 25 (1979).

\bibitem{Goryunov} V.V. Goryunov,
Singularities of projections of complete intersections,
{\it J. Soviet Math.} {27} (1984), 2785--2811.
%[Translated from Itogi Nauki i Tekhniki. Ser. Sovrem. Probl. Mat. 22 (1983),167--206.]
%

\bibitem{Kabata} Y. Kabata,
Recognition of plane-to-plane map-germs,
preprint (2015), ArXiv: 1503.08544.
%
\bibitem{Landis} E. E. Landis,
Tangential singularities, Funt. Anal. Appl. 15 (1981), 103--114 (translation).
%
\bibitem{little} J. A. Little, On the singularities of submanifolds of heigher
dimensional Euclidean space. {\it Annli Mat. Pura et Appl.} (4A) 83 (1969), 261--336.

\bibitem{dmm} D. K. H. Mochida, M. C. Romero-Fuster and M. A. S. Ruas, The geometry
of surfaces in 4-space from a contact viewpoint. {\it Geometria Dedicata} 54 (1995), 323--332.

\bibitem{mond1} D. M. Q. Mond, Classification of certain singularities and applications to differential geometry. Ph.D. thesis, The University of Liverpool, 1982.

\bibitem{mond} D. M. Q. Mond, On the classification of germs of maps from ${{\mathbb R}}^2$ to ${{\mathbb R}}^3$.
{\it Proc. London Math. Soc.} 50 (1985), 333--369.

 \bibitem{mond2}  D. M. Q. Mond. Some remarks on the geometry and classification of germs of maps from
surfaces to 3-space. {\it Topology}. 26 (3) (1987), 361--383.

\bibitem{projsrfR4}J. J. Nu\~no-Ballesteros and F. Tari, Surfaces in $\mathbb R^4$ and their projections to 3-spaces.
{\it Roy. Proc. Edinburgh Math. Soc.} 137A (2007), 1313--1328.



%\bibitem{NT} A. C. Nabarro and F. Tari,
%Families of surfaces and conjugate curve congruences,
%Adv. Geom. 9 (2) (2009), 279--309.
%
%\bibitem{Oliver} J. M. Oliver,
%Binary differential equations with discriminant having a cusp singularity,
%Jour. Dyn. Control Systems, 17 (2) (2011), 207--230.
%
%\bibitem{Olver} P. J. Olver,
%{\em Equivalence, Invariants and Symmetry}, Cambridge Univ. Press (1995).
%
\bibitem{OsetTari} R. Oset-Sinha, F. Tari, Projections of surfaces in $\mathbb R^4$ to $\mathbb R^3$ and the geometry of their singular images.
{\it Rev. Mat. Iberoam. }31, no. 1, 33-50, 2015.
%
\bibitem{Platonova} O. A. Platonova,
Singularities of the mutual disposition of a surface and a line,
Uspekhi Mat. Nauk, 36:1 (1981), 221--222.
%
\bibitem{Platonova2} O. A. Platonova,
Projections of smooth surfaces,
{\it J. Soviet Math. }{35} no.6 (1986), 2796--2808
%[Tr. Sem. I. G. Petrovskii 10 (1984), 135--149 in Russian].
%

%\bibitem{Rieger} J. H. Rieger,
%Families of maps from the plane to the plane,
%J. London Math. Soc.  {\bf 36} (1987), 351--369.
%
%\bibitem{Rieger2} J. H. Rieger,
%Versal topological stratification and the bifurcation geometry of map-germs of the plane,
%Math. Proc. Cambridge Phil.  Soc. 107, no. 1, (1990), 127-147.
%
%\bibitem{Rieger3} J. H. Rieger,
%The geometry of view space of opaque objects bounded by smooth surfaces,
%Artificial Intelligence. 44 (1990), 1-40.
%
%\bibitem{Salmon} G. Salmon,
%{\it A treatise on the analytic geometry of three dimensions},
%4th edition, Dublin (1882).
%
%\bibitem{Tari1} F. Tari,
%Two parameter families of implicit differential equations,
%Discrete and continuous dynamical systems, 13 (2005), 139--262
%
%\bibitem{Tari2} F. Tari,
%Two parameter families of binary differential equations,
%Discrete and continuous dynamical systems, 22 (3) (2008), 759--789.
%
%\bibitem{Tresse} A. Tresse,
%Sur les invariants des groupes continus de transformations,
%Acta Math. 18 (1894), 1--88.
%
%\bibitem{UV1} R. Uribe-Vargas,
%A projective invariant for swallowtails and godrons,
%and global theorems on the flecnodal curve,
%Mosc. Math. J., 6 (2006), 731--772.
%
%\bibitem{UV2} R. Uribe-Vargas,
%Surface evolution, implict differential equations and
%pairs of Legendrian fibrations,
%preprint.
%
%\bibitem{West} J. West, The Differential Geometry of the Cross-Cap,
%Dissertation, University of Liverpool (1995).
%
%\bibitem{Wil} E. J. Wilczynski,
%Projective Differential Geometry of Curved Surfaces,
%{\it Trans. Amer. Math. Soc.} 10 (1909), 279--296.
%
%\bibitem{YKO} T. Yoshida, Y. Kabata and T. Ohmoto,
%Bifurcations of plane-to-plane map-germs of corank $2$,
%Quarterly J. Math. (2014) doi:10.1093/qmath/hau013.
%
%\bibitem{YKO2} T. Yoshida, Y. Kabata and T. Ohmoto,
%Bifurcations of plane-to-plane map-germs of corank $2$ of parabolic type,
%to appear in RIMS koukyuroku Bessatsu (2015).
%
\end{thebibliography}
\end{document}